\documentclass[10pt]{amsart}

\usepackage{amssymb,amsmath,amsthm,amsfonts,microtype}
\usepackage{graphicx}
\usepackage{pdfsync}

\usepackage{tikz,tikz-3dplot}
\usetikzlibrary{shapes,positioning,intersections,quotes}
\usepackage[font=small,labelfont=bf]{caption}

\newcommand{\comment}[1]{}

\newtheorem{lem}{Lemma}
\newtheorem{propn}{Proposition}

\newtheorem{thm}{Theorem}

\theoremstyle{remark}

\theoremstyle{definition}

\newcommand{\R}{\mathbb R}

\newcommand{\Z}{\mathbb Z}
\newcommand{\T}{\mathbb T}

\newcommand{\N}{\mathbb N}

\DeclareMathOperator{\supp}{supp}

\DeclareMathOperator{\lcm}{lcm}

\newcommand{\A}{\alpha}

\newcommand{\la}{\lambda}
\newcommand{\lm}{\lambda}

\newcommand{\be}{\begin{equation}}
\newcommand{\ee}{\end{equation}}
\newcommand{\eq}{\begin{equation}}
\newcommand{\bee}{\begin{equation*}}
\newcommand{\eee}{\end{equation*}}

\begin{document}
\title{The Discrete Spherical Maximal Function: \\ A new proof of $\ell^2$-boundedness}
\author{Neil Lyall \quad\quad  \'{A}kos Magyar  \quad\quad Alex Newman \quad\quad Peter Woolfitt}
\thanks{The first and second authors were partially supported by grants NSF-DMS 1702411 and NSF-DMS 1600840, respectively.}

\address{Department of Mathematics, The University of Georgia, Athens, GA 30602, USA}
\email{lyall@math.uga.edu}
\email{magyar@math.uga.edu}
\email{alxjames@uga.edu}
\email{pwoolfitt@uga.edu}


\begin{abstract} 
We provide a new direct proof of the $\ell^2$-boundedness of the Discrete Spherical Maximal Function that neither relies on abstract transference theorems (and hence Stein's Spherical Maximal Function Theorem) nor on delicate asymptotics  for the Fourier transform of discrete spheres.
 \end{abstract}
\maketitle

\setlength{\parskip}{3pt}

\section{Introduction}

The study of discrete analogues of central constructs of Euclidean harmonic analysis, initiated by Bourgain \cite{B1}, has grown into a vast, active area of research. An important result in this development is the $\ell^p$-boundedness of the so-called discrete spherical maximal function \cite{MSW}.

Beyond its own intrinsic interest, this operator, or more precisely certain ``mollified variants", play a crucial role in studying certain geometric point configurations in positive density subsets of the integer lattice, see \cite{LMdist}.

Let $d\geq 5,\,\la^2\in\N$, and $N_\la :=|\{m\in\Z^d:\ |m|=\la\}|.$ It is well-known, see for example \cite{V}, that \[c_d\la^{d-2}\leq N_\la \leq C_d\la^{d-2}\] for some constants $0<c_d<C_d$. 
For $f:\Z^d\to\R$ define the discrete spherical averages
\bee
A_\la f(n)= N_\la^{-1} \sum_{|m|=\la} f(n-m)
\eee
and the maximal operator
\bee
A_\ast f(n)= \sup_{\la} |A_\la f(n)|.\eee

The variables $n,m$ in the two equations above, and throughout this short note, are always assumed to be in $\Z^d$. Furthermore, the parameter $\la$ will always be assumed to satisfy $\la^2\in\N$. 

In \cite{MSW} it was shown that for $p>d/(d-2)$ one has the estimate
\bee
\|A_\ast f\|_p \leq C_{p,d}\,\|f\|_p
\eee
where $\|f\|_p = (\sum_x |f(x)|^p)^{1/p}$ 
denotes the $\ell^p(\Z^d)$ norm of the function $f$.
It was further noted in \cite{MSW} that the condition that $d\geq 5$ and $p>d/(d-2)$ are both sharp.

The approach taken in \cite{MSW} had three main steps. The first step was to approximate $A_\lm$ by an infinite sum of simpler operators $M_\lm^{a/q}$, each associated to a reduced fraction $a/q$, with $0<a/q\leq1$. A general abstract transference theorem, which allows one to pass from certain convolution operators on $\R^d$ to analogous operators on $\Z^d$, was then used to analyze each $M_\lm^{a/q}$. In particular, this approach makes use of Stein's Spherical Maximal Function Theorem  \cite{St}. The final step of the argument is to show that the approximation taken in the first step is adequate, this step uses the full asymptotic expansion for the Fourier transform of (the indicator function of) the discrete sphere of radius $\lm$ in $\Z^d$.

In this note we provide a short direct proof  of the $\ell^2$ case of the main result in \cite{MSW}. Our direct proof relies on the observation that one obtains gains in $\ell^2$ for maximal operators at a single dyadic scale, when applied to functions whose Fourier transform is suitably localized away from rational points with suitably small denominators, specifically Proposition \ref{MainProp} below. This combined with an almost orthogaonality argument quickly leads to the proof Theorem 1 below. Note that we do not need the full asymptotic expansion of the underlying multipliers neither any transference arguments to utilise Stein's spherical maximal theorem.

Our main result is the following,

\begin{thm}\label{Thm1}
If $d\geq5$, then
\bee
\|A_\ast f\|_2 \leq C_{d}\,\|f\|_2.
\eee
\end{thm}

\section{Key estimates for maximal operators at a single dyadic scale}

Recall that for  $f\in \ell^1(\mathbb{Z}^d)$ we define its \textit{Fourier transform} $\widehat{f}: \mathbb{T}^d\to\mathbb{C}$ by
$$\widehat{f}(\A)=\sum_{n\in\mathbb{Z}^d}f(n)e^{-2\pi i n\cdot\A}.$$

Before stating Proposition \ref{MainProp} we need to introduce some additional notation. For any  integer $j\geq0$ we let
$
q_j=\lcm\{1, 2, \dots,2^j\}
$
and note that $q_j\asymp e^{2^j}$. For any non-negative integers $j$ and $k$ that satisfy $2^j\leq k$ , we let
\be
\Omega_{j,k}:=\{\A\in\mathbb{T}^d: \A \in[-2^{j-k},2^{j-k}]^d+(q_j^{-1}\mathbb{Z})^d\}.
\ee

\smallskip

 \begin{propn}\label{MainProp}
 If $d\geq 5$, $k\in\mathbb{N}$, and $1\leq j\leq \log_2(k)-2$, then one has the estimate
 \begin{equation}
     \Bigl\|\sup_{2^k\leq\lambda\leq2^{k+1}} |A_\la f|\Bigr\|_2\ll 2^{-j/2}j^{-1}\|f\|_2
 \end{equation}
whenever $\supp \widehat{f}\subseteq \Omega_{j,k}^c$, where $\Omega_{j,k}^c$ denotes the complement of $\Omega_{j,k}$.
 \end{propn}
 
In the Proposition above, and for the rest of this short note, we use the notation $A\ll B$ to denote that $A\leq C B$ for some constant $C$ that may depend on $d$, which we consider fixed and greater than or equal to 5.
 
The proof of Proposition \ref{MainProp} is presented in Section \ref{PropProof}, while the reduction of Theorem \ref{Thm1} to Proposition \ref{MainProp} is presented in Section \ref{Thm1Proof} below. 
We conclude this section by noting that Proposition \ref{MainProp} immediately implies the following ``mollified variant" of Theorem \ref{Thm1} which is of independent interest.

 \begin{thm}\label{Thm2}
If $d\geq5$, $\eta>0$, and $L\geq q_\eta^4$, then one  has the estimate
\be\label{3}
\Bigl\|\sup_{\lm\geq\eta^{-2} L}|A_\la f|\Bigr\|_2\ll \eta\, \|f\|_2
\ee
whenever $\supp \widehat{f}\subseteq \Omega_{\eta,L}^c$, with $ \Omega_{\eta,L} = \{\A\in\mathbb{T}^d: \A\in[-L^{-1},L^{-1}]^d+(q_\eta^{-1}\mathbb{Z})^d\}$ and $q_\eta=\lcm\{1\leq q \leq \eta^{-2}\}$.
\end{thm}

Indeed, note that in proving \eqref{3} one may restrict the sup to $\eta^{-2} L\leq\lm\leq 2\eta^{-2} L$. Choosing $k,j\in\mathbb{N}$ such that $2^k\leq\eta^{-2} L\leq 2^{k+1}$ and $2^j\geq \eta^{-2}$ we have that  $2^{k-j}\leq L$ and hence $\Omega_{j,k} \subseteq \Omega_{\eta,L}$. 
Applying Proposition \ref{MainProp} with $j$ and $k$ chosen as above implies Theorem \ref{Thm2}. 

This provides a slight strengthening of Proposition 5 in \cite{LMdist}, more importantly it provides a significantly simpler direct proof.

\section{Proof of Theorem \ref{Thm1}}\label{Thm1Proof}


\subsection{A smooth sampling function supported on $\Omega_{j,k}$}

Let $\psi \in \mathcal{S}(\mathbb{R}^d)$ be a Schwartz function satisfying \begin{equation*}
    1_{Q}(\xi)\leq \widetilde{\psi}(\xi)\leq 1_{2Q}(\xi)
\end{equation*}
where $Q=[-1/2,1/2]^d$ and
$$\widetilde{\psi}(\xi):=\int_{\mathbb{R}^d}\psi(x) e^{-2\pi i x\cdot\xi}dx$$
denote the Fourier transform of $\psi$ on $\mathbb{R}^d$.
For a given $q \in \mathbb{N}$ and $L>q$ we define $\psi_{q,L}:\mathbb{Z}^d\to \mathbb{R}$ as
\begin{equation}
    \psi_{q,L}(m)=
 \begin{cases} 
      \left(\frac{q}{L}\right)^d\psi\left(\frac{m}{L}\right) & \textrm{if } m\in(q\mathbb{Z})^d  \\
      0 &  \textrm{otherwise}
 \nonumber
   \end{cases}
\end{equation}

Writing $m=qr+s$ with $r\in\mathbb{Z}^d$ and $s\in\mathbb{Z}^d/q\mathbb{Z}^d$, it follows from Poisson summation that 
\bee
\widehat{\psi}_{q,L}(\A) =  \sum_{m\in\mathbb{Z}^d}\psi(m)e^{-2\pi i m\cdot\A}
\eee
is a $q^{-1}$-periodic function on $\T^d$ that satisfies
\bee
\widehat{\psi}_{q,L}(\A) =\sum_{\ell\in\mathbb{Z}^d}\widetilde{\psi}(L(\A- \ell/q)).
\eee

For a given $k\in\mathbb{N}$ and $0\leq j\leq J_k:=[\log_2(k)]-2$, we now define the sampling function
\be
\Psi_{j,k}=\psi_{q_j,2^{k-j}}
\ee
and note that $\supp{\widehat{\Psi}_{j,k}}\subseteq \Omega_{j,k}$. 

Finally we define $\Delta\Psi_{j,k}= \Psi_{j+1,k}-\Psi_{j,k}$ and note the important almost orthogonality property they enjoy.

 \begin{lem}\label{orthoK} There exists a constant $C=C_\Psi>0$ such that 
 \[\sum_{k\geq2^j}|\widehat{\Delta\Psi}_{k,j}(\A)|^2\leq C\] uniformly in $j\in \N$ and $\A\in\mathbb{T}^d.$
  \end{lem}
  \begin{proof}[Proof of Lemma \ref{orthoK}]
 
Note that $\Omega_{k+1,j}\subseteq \Omega_{k,j}.$  Now fix $j\in\mathbb{N}$.
If $\A \notin \Omega_{2^j,j}$, then $\widehat{\Delta\Psi}_{k,j}(\A)=0$. 

If $\A\in \Omega_{2^j,j}$, then we define $k_1=k_1(j):=\max\{k\geq2^j:\A\in \Omega_{k,j}\}$.
  Then there exists a unique $\ell_1\in\mathbb{Z}^d$ such that $|\A-\ell_1/q_j|\leq2^{j-k_1}.$  
Clearly  $\widehat{\Delta\Psi}_{k,j}(\A)=0$ if $k>k_1$, while if $2^j\leq k \leq k_1$ we have 
   $\widehat{\Psi}_{k,j}(\A)=\tilde{\Psi}(2^{k-j}(\A-\ell_1/q_j))$. It therefore follows, by writing $\Delta\Psi_{k,j}=(\Psi_{k,j+1}-1)+(1-\Psi_{k,j})$,  that
\[
 |\widehat{\Delta\Psi}_{k,j}(\A)|\leq C_\Psi\,2^{k-j}|\A-\ell_1/q_j|\leq C_\Psi\,2^{k-k_1}
 \]
 and hence that 
 \[  \sum_{k\geq2^j}|\widehat{\Delta\Psi}_{k,j}(\A)|^2\leq C_\Psi\,\sum_{1\leq k\leq k_1}2^{-2(k_1-k)}\leq \frac{4}{3}\,C_\Psi.\qedhere\] 
   \end{proof}
   
\subsection{Proof that Proposition \ref{MainProp} implies Theorem \ref{Thm1}}

Let 
\be
M_kf:=\sup_{2^k\leq\lambda\leq2^{k+1}}|A_\la f|.
\ee

Writing
\[f= f*\Psi_{k,0}+\sum_{j=0}^{J_k-1}f*\Delta\Psi_{k,j}+(f-f*\Psi_{k,J_k})\]
it follows by subadditivity that
\begin{equation}\label{subadd}
M_k f\leq M_k(f*\Psi_{k,0})+\sum_{j=0}^{J_k-1}M_k(f*\Delta\Psi_{k,j})+M_k(f-f*\Psi_{k,J_k})
\end{equation}

Theorem \ref{Thm1} will now follow from a few observations and applications of Proposition \ref{MainProp}, in light of the fact that \[A_*f=\sup_k M_kf.\]

First we note that it is straightforward to verify that the first term on the right in  (\ref{subadd}) above satisfies
\[M_k(f*\Psi_{k,0})\leq C_\Psi \mathcal{H}f\]
uniformly in $k$, where
\[\mathcal{H}f(n)=\sup_{\ell>0}\frac{1}{(2\cdot2^\ell+1)^d}\Bigl|\sum_{m\in[-2^\ell,2^\ell]^d\cap\mathbb{Z}^d}f(n-m)\Bigr|\]
denotes the discrete Hardy-Littlewood maximal operator. Since, by the same arguments as in Euclidean spaces, we have $\|\mathcal{H}f\|_2\ll\|f\|_2$, it follows that
\[\sup_k\|M_k(f*\Psi_{k,0})\|_2\ll\|f\|_2.\]

For the middle terms in (\ref{subadd}) we first note that 
    \bee
        \sup_{k}\sum_{j=0}^{J_k-1}M_k(f*\Delta\Psi_{k,j})\ll\Bigl(\sum_{k=0}^\infty\Bigl|\sum_{j=0}^{J_k-1}M_k(f*\Delta\Psi_{k,j})\Bigr|^2\Bigr)^{1/2}\label{supto2norm}
    \eee
    
    Taking $\ell^2$ norms of both sides of the inequality above and applying Minkowski's  inequality, followed by an application of Proposition \ref{MainProp}, gives    
  \bee  \begin{aligned}
    \Bigl\|\sup_{k}\sum_{0\leq j\leq J_k}M_k(f*\Delta\Psi_{k,j})\Bigr\|_2
    &\leq\sum_j\Bigl(\sum_{k\geq 2^j}\|M_k(f*\Delta\Psi_{k,j})\|_2^2\Bigr)^{1/2}\\
&\ll\sum_j 2^{-j/2}\Bigl(\sum_{k\geq 2^j}\|f*\Delta\Psi_{k,j}\|^2_2\Bigr)^{1/2}\ll \|f\|_2\label{ortho}
    \end{aligned}\eee
where the last inequality above follows from Lemma \ref{orthoK}.
    
One more application of Proposition \ref{MainProp} with $j=[\log_2k]-2$ to the last term in (\ref{subadd}) gives
\[\pushQED{\qed} 
\Bigl\|\sup_k M_k(f-f*\Psi_{k,J_k})\Bigr\|_2
\leq\Bigl(\sum_{k=1}^{\infty}\|M_k(f-f*\Psi_{k,J_k})\|_2^2\Bigr)^{1/2}
\ll \Bigl(\sum_{k=1}^\infty k^{-1}(\log_2k)^{-2} \Bigr)^{1/2}\|f\|_2\ll \|f\|_2.
\qedhere \popQED\]

\smallskip

\section{Proof of Proposition \ref{MainProp}}\label{PropProof}

Fix $\varepsilon=2^{-2k}$. We start by observing that
\begin{equation*}
1_{\{|m|=\lm\}}(m)=\int_0^1 e^{2\pi i(|m|^2-\lambda^2)t}dt= e^{2\pi\varepsilon\lambda^2}\int_0^1 e^{2\pi i|m|^2(t+i\varepsilon)}e^{-2\pi i \lambda^2 t}\,dt
\end{equation*}
where $1_{\{|m|=\lm\}}$ denotes the indicator function of the discrete sphere of radius $\lm$ in $\Z^d$.

Since $N_\lm \asymp\lm^{d-2}$ it therefore follows that
\begin{equation*}
M_kf\ll \sup_{2^{k}\leq\lambda\leq 2^{k+1}} \frac{1}{\lambda^{d-2}}\int_0^1|f*s_t|\,dt
\end{equation*}
where $s_t(m)=e^{2\pi i|m|^2(t+i\varepsilon)}$, and hence that
\begin{equation*}
\|M_kf\|_2\ll \varepsilon^{(d-2)/2} \int_0^1\|f*s_t\|_2\,dt\leq \varepsilon^{(d-2)/2}\left(\int_0^1\|\widehat{s_t}\, 1_{\Omega^c_{j,k}}\|_\infty \,dt\right)\|f\|_2.
\end{equation*} 

Thus, in order to prove Proposition \ref{MainProp} it suffices to show that 
\begin{equation}\label{StBound}
    \int_0^1\|\widehat{s_t}\, 1_{\Omega^c_{j,k}}\|_\infty \,dt\ll \varepsilon^{-(d-2)/2}2^{-j/2}j^{-1}.
\end{equation}

To do this we will employ the circle method and decompose the interval into Farey arcs, that is neighborhoods  $V_{a,q}$ of reduced rationals $a/q$ which allows us to estimate $\widehat{s}_t(\xi)$ by using Poisson summation and properties of Gaussian sums. Specifically, we
decompose the interval $[0,1]$ into neighborhoods of rationals whose denominator is smaller than $2^k$ as follows:
Let
    $$H=\{a/q : 1\leq q\leq 2^k, 0<a\leq q, (a,q)=1\}$$
and define  $$V_{a,q}=\left\{t\in [0,1] : \left|t-a/q\right|=\min_{r\in H} |t-r| \right\}.$$
    
Note that, by Dirichlet's principle, for every $t \in [0,1]$, there exists $a/q \in H$ such that
    $|t-a/q|\leq 2^{-k}q^{-1}
    $, thus we have that $|V_{a,q}|\leq 2^{-k+1}q^{-1}$. Also, if $a/q\neq a'/q'$ with $(a,q)=(a',q')=1$ and $1\leq q,q'\leq 2^k$ then $|a/q-a'/q'|\geq 1/(qq')\geq 2^{-k}q^{-1}$, hence $|V_{a,q}|\geq 2^{-k}q^{-1}$. Thus the Farey arcs $V_{a,q}$ 
at level $2^k$  provide a partition (up to endpoints) of [0,1] into intervals of length $|V_{a,q}|\approx 2^{-k}q^{-1}$.

It follows from Poisson summation that for $1\leq a\leq q$, $(a,q)=1$, $1\leq q \leq 2^k$ one has
\begin{equation}\label{StEstimate}
|\widehat{s_t}(\A)|\leq q^{-d/2}(\varepsilon+|\tau|)^{-d/2}\sum_{\ell\in\mathbb{Z}^d}e^{-\frac{\pi}{2}|\A-\ell/q|^2/({\varepsilon+\varepsilon^{-1}|\tau|^2})}
\end{equation}
for each $t\in V_{a,q}$ with $\tau=t-a/q$.
The details of the calculation to derive estimate (\ref{StEstimate}) are laid out more carefully in \cite{Magyar}, but they can be briefly summarize as follows: First write $\widehat{s_t}$ as a product of one dimensional functions. An application of Poisson summation and a change of variables leaves a double sum that can be recognized as a quadratic Gaussian sum, which can be bounded by $q^{-1/2}$, and a sum of terms involving $\widetilde{s}_\tau(\ell/q-\A)$ which has a simple closed form. 
See formula $(12)$ in \cite{Magyar}.

Since $|\tau|\leq 2^{-k}q^{-1}$, it follows that $q^2(\varepsilon+\varepsilon^{-1}|\tau|^2)\ll 1$, and hence that
\[\sum_{\ell\in\mathbb{Z}^d}e^{-\frac{\pi}{2}|\A-\ell/q|^2/({\varepsilon+\varepsilon^{-1}|\tau|^2})}\ll1\]
which in turn implies that if $t\in V_{a,q}$ with $t=a/q+\tau$, then 
\begin{equation}\label{CrudeEst}
\|\widehat{s_t}\|_\infty\ll q^{-d/2}(\varepsilon+|\tau|)^{-d/2}.
\end{equation}

We write 
\be\label{divide}
    \int_0^1\|\widehat{s_t}\chi_{_{\Omega^c_{j,k}}}\|_\infty\,dt \,=\, \sum_{q|q_j}\sum_{(a,q)=1}\int_{V_{a,q}}\|\widehat{s_t}\,1_{\Omega^c_{j,k}}\|_\infty dt
      \ + \ \sum_{q\nmid q_j}\sum_{(a,q)=1}\int_{V_{a,q}}\|\widehat{s_t}\,1_{\Omega^c_{j,k}}\|_\infty \,dt.
\ee

In order to estimate the first double sum above we consider separately the case when $|\tau|\geq2^{j/2}\varepsilon$ and $|\tau|\leq2^{j/2}\varepsilon$.
When $|\tau|\geq2^{j/2}\varepsilon$ we  use estimate (\ref{CrudeEst}) to bound it by
\be\label{easy}
\begin{aligned}
\sum_{q| q_j}\sum_{(a,q)=1}\int_{V_{a,q}}q^{-d/2}(\varepsilon+|\tau|)^{-d/2} dt
&\ll\sum_{q|q_j}q^{-d/2+1}\int_{\varepsilon2^{j/2}}^\infty (\varepsilon+|\tau|)^{-d/2}d\tau\\
&\ll \varepsilon^{-(d-2)/2}2^{-j(d-2)/4} \,\sum_{q|q_j}q^{-d/2+1}.
\end{aligned}
\ee

When $|\tau|\leq2^{j/2}\varepsilon$ we note that because $q|q_j$ we have that
\smallskip
\be\label{sum2}
    \sum_{\ell\in\mathbb{Z}^d}e^{-\frac{\pi}{2}|\A-\ell/q|^2/(\varepsilon+\varepsilon^{-1}|\tau|^2)}=e^{-\frac{\pi}{2}|\A-\ell_0/q|^2/(\varepsilon+\varepsilon^{-1}|\tau|^2)}
    + \sum_{\ell\neq\ell_0}e^{-\frac{\pi}{2}|\A-\ell/q|^2/(\varepsilon+\varepsilon^{-1}|\tau|^2)}
\ee
where $\ell_0$ denotes the nearest integer to $q\A$. For every $\A\in \Omega^c_{j,k}$
 we have
$|\A-\frac{\ell_0}{q}|=|\A-\frac{q_j\ell_0/q}{q_j}|\geq2^{j-k}$ and hence that
\be\label{single}
|e^{-\frac{\pi}{2}|\A-\ell_0/q|^2/(\varepsilon+\varepsilon^{-1}|\tau|^2)}|\ll e^{-c(2^{j-k})^2/2^{j-2k}}\ll e^{-c2^j}
\ee
since $\varepsilon+\varepsilon^{-1}|\tau|^2\leq2\cdot2^j\varepsilon\ll 2^{j-2k}$.
To estimate the sum where $\ell\neq\ell_0$ in (\ref{sum2}) above
we again use the fact that $\varepsilon+\varepsilon^{-1}|\tau|^2\ll 2^{j-2k}$. Since $|q\A-\ell|\geq 1/2$ for $\ell\neq\ell_0$ and $q|q_j$ with $j\leq \log_2 k -2$ it follows that
\begin{equation*}q^2(\varepsilon+\varepsilon^{-1}|\tau|^2)\ll(2^{2^j})^22^{j-2k}\leq2^{k}2^{-2k}\leq2^{-k}
\end{equation*}
and hence that
\be\label{not}
\sum_{\ell\neq\ell_0}e^{-\frac{\pi}{2}|\A-\ell/q|^2/(\varepsilon+\varepsilon^{-1}|\tau|^2)}\ll 2^{-c2^{k}}.
\ee

Combining estimates (\ref{single}), and (\ref{not}) it follows that
\[\|\widehat{s_t}\,1_{\Omega^c_{j,k}}\|_\infty \ll q^{-d/2}(\varepsilon+|\tau|)^{-d/2}\,(e^{-c2^j}+2^{-c2^{k}})
\]
whenever $|\tau|\leq2^{j/2}\varepsilon$ and $q|q_j$. Further combining this with (\ref{easy}) we obtain that
\begin{align*}
    \sum_{q|q_j}\sum_{(a,q)=1}\int_{V_{a,q}}\|\widehat{s_t}\,1_{\Omega^c_{j,k}}\|_\infty dt&\ll\sum_{q|q_j}q^{-d/2+1}\varepsilon^{-(d-2)/2}\left[(e^{-c2^j}+2^{-c2^{k}})+2^{-j(d-2)/4} \right]\\
    &\ll \varepsilon^{-(d-2)/2}\,2^{-j(d-2)/4}\,\sum_{q|q_j}q^{-d/2+1}\\
    &\ll \varepsilon^{-(d-2)/2}\,2^{-j(d-2)/4}
\end{align*}
as the sum over $q$ converges for $d\geq 5.$

 In order to estimate the second double sum in (\ref{divide}) we need the following observation, whose proof we delay until after completing the proof of Proposition \ref{MainProp}.

\begin{lem}\label{Lemma:qdoesnotdivide}
Given $r>1$, $k \in \mathbb{N}$,  then
\[\sum_{q\nmid Q_k}q^{-r}\ll_r k^{-r+1}(\log k)^{-1}\]
where $Q_k=\mathrm{lcm} \{1\leq q\leq k\}$ and $\ll_r$ denotes less than a constant depending on $r$.
\end{lem}

Using estimate (\ref{CrudeEst}) and Lemma \ref{Lemma:qdoesnotdivide} one can bound the second double sum in (\ref{divide}) above by
\bee
\sum_{q\nmid q_j}\sum_{(a,q)=1}\int_{V_{a,q}}q^{-d/2}(\varepsilon+|\tau|)^{-d/2} dt
\ll \varepsilon^{-(d-2)/2}\sum_{q\nmid q_j}q^{-d/2+1}
\ll \varepsilon^{-(d-2)/2}2^{-j(d-4)/2}j^{-1}
\eee
whenever $d\geq 5$ completing the proof of Proposition \ref{MainProp}.\qed

\begin{proof}[Proof of Lemma \ref{Lemma:qdoesnotdivide}]
If $q\in\mathbb{N}$ such that $q\nmid Q_k$, then either a large power of a small prime divides $q$, or a large prime divides $q$. Explicitly, for prime $p\le k$, let
$a_p=\min\{a\in\mathbb{N}~:~k<p^a\}.$
If $q\nmid Q_k$ then one of the following must occur:
\begin{enumerate}
\item[(i)] there exists a $p>k$ such that $q=pq_1$
\item[(ii)] there exists a $p<k$ such that $q=p^{a_p}q_1$.
\end{enumerate}

In the first case
\begin{align*}
\sum_{\text{(i) holds}}q^{-r}&\le\sum_{p>k}\sum_{q_1\in\mathbb{N}}p^{-r}q_1^{-r}
\\&\ll_r\sum_{p>k}p^{-r}
\\&\ll_r\sum_{m\ge0}\sum_{p\in [2^mk,2^{m+1}k)}2^{-mr}k^{-r}
\\&\ll_r\sum_{m\ge0}2^{-mr}k^{-r} \frac{2^mk}{m\log(k)}
\ll_r k^{-r+1}(\log k)^{-1}.
\end{align*}

In the second case
\[
\sum_{\text{(ii) holds}}q^{-r}\le \sum_{p\le k}p^{-ra_p}\sum_{q_1\in\mathbb{N}}q_1^{-r}\ll_r \frac{k}{\log k} k^{-r}.\qedhere \]
\end{proof}



\end{document}